\documentclass{article}
\usepackage[utf8]{inputenc}
\usepackage{enumerate}
\usepackage{amsmath}
\usepackage{amssymb,latexsym}
\usepackage{amsthm}
\usepackage{color}
\usepackage{graphicx}
\usepackage{amscd}
\usepackage{hyperref}
\usepackage{multicol}
\usepackage{textgreek}
\usepackage{soul}
\newcommand{\F}{\mathbb{F}}

\newcommand{\cB}{\mathcal{B}}
\newcommand{\cA}{\mathcal{A}}
\newcommand{\cM}{\mathcal{M}}
\newcommand{\Fq}{\F_{q}}
\newcommand{\Fqn}{\F_{q^n}}
\newcommand{\Fqt}{\F_{q^t}}
\newcommand{\la}{\langle}
\newcommand{\ra}{\rangle}
\newcommand{\homm}{\End_{\Fq}(\Fqn)}
\DeclareMathOperator{\PG}{{PG}}
\DeclareMathOperator{\Aut}{{Aut}}
\DeclareMathOperator{\GL}{{GL}}
\DeclareMathOperator{\PGL}{{PGL}}

\DeclareMathOperator{\GaL}{\Gamma L}
\DeclareMathOperator{\Gal}{Gal}
\DeclareMathOperator{\rk}{rk}
\DeclareMathOperator{\N}{N}
\DeclareMathOperator{\End}{End}
\DeclareMathOperator{\cC}{\mathcal{C}}

\DeclareMathOperator{\cS}{\mathcal{S}}
\DeclareMathOperator{\diag}{diag}
\newtheorem{theorem}{Theorem}[section]
\newtheorem{lemma}[theorem]{Lemma}
\newtheorem{corollary}[theorem]{Corollary}
\newtheorem{proposition}[theorem]{Proposition}

\theoremstyle{definition}
\newtheorem{definition}[theorem]{Definition}
\newtheorem{example}[theorem]{Example}

\newtheorem{remark}[theorem]{Remark}

\title{A standard form for scattered linearized polynomials and properties 
of the related translation planes}
\author{Giovanni Longobardi \quad Corrado Zanella}
\date{}

\begin{document}

\maketitle

\begin{abstract}
In this paper we present results concerning the stabilizer $G_f$ in $\GL(2,q^n)$ 
of the subspace $U_f=\{(x,f(x))\colon x\in\Fqn\}$, $f(x)$ a scattered linearized polynomial 
in $\Fqn[x]$. 
Each $G_f$ contains the $q-1$ maps $(x,y)\mapsto(ax,ay)$, $a\in\Fq^*$. 
By virtue of the results of Beard~\cite{beard1,beard2} and Willett~\cite{willett}, the matrices 
in $G_f$ are simultaneously diagonalizable. 
This has several consequences: $(i)$~the polynomials such that $|G_f|>q-1$ have a standard form 
of type $\sum_{j=0}^{n/t-1}a_jx^{q^{s+jt}}$ for some $s$ and $t$ such that $(s,t)=1$, $t>1$ a divisor 
of $n$; $(ii)$~this standard form is essentially unique; $(iii)$~for $n>2$ and $q>3$, the translation plane $\cA_f$ 
associated with $f(x)$ admits nontrivial affine homologies if and only if $|G_f|>q-1$, and in that case 
those with axis through the origin form two groups of cardinality $(q^t-1)/(q-1)$ that 
exchange axes and coaxes; $(iv)$~no plane of type $\cA_f$, $f(x)$ a scattered polynomial not of
pseudoregulus type, is a generalized André plane. 
\end{abstract}



\section{Introduction}

A \textit{scattered polynomial} is an  $\Fq$-linearized polynomial $f(x)=\sum_{i=0}^{n-1}a_ix^{q^i}\in\Fqn[x]$ such that 
for any $y,z\in\Fqn$ the condition $zf(y)-yf(z)=0$ implies that $y$ and $z$ are $\Fq$-linearly 
dependent.
Scattered polynomials have been investigated for about 20 years because of their interest in various 
combinatorial contexts, such as blocking sets, planar spreads and the related translation planes \cite{tp1,LP2001}, 
and rank distance codes (\textit{RD codes}), see \cite{PolverinoZulloConnections} for a survey. 
The RD codes are codes in which words are matrices and the distance between two matrices $A$ and $B$ is 
$\rk(A-B)$. 
For an RD code $\cC$ in $\Fq^{m\times n}$, an inequality analogous to Singleton's one holds: 
\begin{equation}\label{eq:singleton}
\log_q|\cC|\le \max\{m,n\}(\min\{m,n\}-d+1)
\end{equation}
where $d$ is the minimum distance of $\cC$.
As it follows immediately from the definition of a scattered polynomial $f(x)$, 
the set $\cC_f=\la x,f(x)\ra_{\Fqn}$ (seen as a subset of $\homm$) is a set of $\Fq$-linear endomorphisms of $\Fqn$ having rank at least $n-1$ and therefore the equality holds in \eqref{eq:singleton}. 
In other words, $\cC_f$ is a \textit{maximum rank distance code}, or MRD code, see \cite{Sh2}.
One aspect of scattered polynomials that makes them interesting objects of research is that they are 
rare.  
In fact, there are only three families of scattered polynomials defined for infinite values of $n$: 
those of pseudoregulus type~\cite{BL2000,LMPT}, 
the Lunardon-Polverino polynomials~\cite{LP2001, Lunardon2015GeneralizedTG, LuTrZh17,Sh}, and the polynomials 
described in \cite{LoMaTrZh21, LoZa21,NeSaZu21}. See \cite{BaMo, BaZh} for asymptotic results in this direction.
In Section \ref{s:diag} of this paper we derive some results on the stabilizer in $\GL(2,q^n)$ 
of the $\Fq$-subspace $U_f=\{(x,f(x))\colon x\in\Fqn\}$ associated with a scattered polynomial $f(x)$, 
which are essential for the main results contained in the following sections. 
In Section \ref{s:sf}, we consider the class $\cS_{n,q}$ of the scattered polynomials such that
the stabilizer of $U_f$ has size greater than $q-1$.
Any polynomial in $\cS_{n,q}$ admits a \textit{standard form}
$h(x)=\sum_{j=0}^{n/t-1}c_jx^{q^{s+jt}}$, $t>1$, i.e.
for any $f(x)\in\cS_{n,q}$ the subspace $U_f$ is in the same orbit of $U_h$ under 
the action  of $\GL(2,q^n)$. Furthermore, such standard form is essentially unique.
In Section \ref{s:tp}, we consider the translation planes associated with scattered polynomials, as 
defined in \cite{tp1}. 
We find that if $f(x)$ is not in $\cS_{n,q}$, then the plane $\cA_f$  has no affine central collineations.
If on the contrary $f(x)\in\cS_{n,q}$, then there are only two centers of affine homologies of $\cA_f$.
The related two groups of homologies in $\GL(2,q^n)$ are such that those in one group have axes and coaxes that are respectively coaxes and axes of the homologies in the other group.

\section{Preliminaries and notation}
In this section, some  fundamental notions about linear sets of the finite projective line are recalled. The reader may refer to the surveys \cite{PolverinoZulloConnections,Polverinosurvey}.
Throughout this paper $q$ denotes a power of a prime $p$. Moreover, in the following for a set $S$ of field elements (or vectors), we denote by $S^*$ the set of non-zero
elements (non-zero vectors) of $S$.
Let $r,n\in\mathbb N$, $r>0$, $n>1$.
Let $U$ be an $r$-dimensional $\Fq$-subspace of the $(2n)$-dimensional vector space
$V(\Fqn^2,\Fq)$.
The following subset of $\PG(\Fqn^2,\Fqn)=\PG(1,q^n)$
\[
L_U=\{\la v\ra_{\Fqn}\colon v\in U^*\}
\]
is called \emph{$\Fq$-linear set} (or just \emph{linear set}) of \emph{rank $r$}.
The linear set $L_U$ is \emph{scattered} if it has the maximum possible size related to the
given $q$ and $r$, that is, $|L_U|=(q^r-1)/(q-1)$.
Equivalently, $L_U$ is scattered if and only if
\begin{equation}\label{e:scattered}
  \dim_{\Fq}\left(U\cap\la v\ra_{\Fqn}\right)\le1\quad\mbox{ for any }v\in\Fqn^2.
\end{equation}
Any $\Fq$-subspace $U$ satisfying \eqref{e:scattered} is a \emph{scattered $\Fq$-subspace}.

Clearly, any $\Fq$-linear set in $\PG(1,q^n)$ of rank greater than $n$ coincides
with $\PG(1,q^n)$.
So, the linear sets of rank $n$ are called linear sets of \emph{maximum rank}.
They are associated with linearized polynomials.
An \emph{$\Fq$-linearized polynomial}, or \emph{$q$-polynomial}, in $\Fqn[x]$ is of type 
$f(x)=\sum_{i=0}^ka_ix^{q^i}$ ($k\in\mathbb N$).
If $a_k\neq0$, then $k$ is the \emph{$q$-degree} of $f(x)$.
It is well-known that the $\Fq$-linearized polynomials of $q$-degree less than $n$ in $\Fqn[x]$
are in one-to-one correspondence with the endomorphisms of the vector space
$V(\Fqn,\Fq)$, see \cite[Chapter 3] {finitefields} and \cite[Subsection 3.1] {PolverinoZulloConnections}.\\ 
Let $f(x)\in\Fqn[x]$ be an $\Fq$-linearized polynomial and define 
\[U_f=\{(x,f(x))\colon x\in\Fqn\},\]
and $L_f=L_{U_f}$.
Such $L_f$ is an $\Fq$-linear set of maximum rank of $\PG(1,q^n)$, and is scattered if and only if
$f(x)$ is.
Two $\Fq$-linearized polynomials $f(x)$ and $g(x)$ in $\Fqn[x]$ are said to be \emph{$\GL$-equivalent}
(or \emph{$\GaL$-equivalent}) when a $\phi$  in $\GL(2,q^n)$ (resp.\ in $\GaL(2,q^n)$) exists such that
$U_f^\phi=U_g$.

Since $\PGL(2,q^n)$ acts $3$-transitively on $\PG(1,q^n)$,
any linear set of maximum rank is projectively equivalent to an $L_U$ such that
$\la(0,1)\ra_{\Fqn}\notin L_U$.
Therefore, a linearized polynomial $f(x)$ exists such that $U=U_f$.

By abuse of notation, $L_f$ will also denote the set $\{f(x)/x\colon x\in\Fqn^*\}$ of the
nonhomogeneous projective coordinates of the points belonging to the set $L_f$.

We will now introduce some more elements that will be investigated in this paper. For more details on the relationship between scattered polynomials and translation planes see \cite{tp1}.
We adopt the notation $G_{\{T\}}$ for the setwise stabilizer of $T\subseteq S$ under the action
of a group $G$ acting on $S$.
Let $G_f=\GL(2,q^n)_{\{U_f\}}$, and
$G_f^\circ=G_f\cup\{O\}$, where $O$ is the zero $2\times2$ matrix.


\section{\texorpdfstring{The stabilizer of a scattered subspace of $\Fqn^2$}{The stabilizer of a scattered subspace}}\label{s:diag}

\subsection{Algebraic properties and representation as right idealizer}

Regarding main definitions on MRD codes, the reader can refer to \cite{LoMaTrZh21,PolverinoZulloConnections, Sh2}. Below we will recall just a few definitions, useful for understanding the next sections. Let $\cC \subseteq \F_{q}^{m \times n}$ be a rank distance code, its \textit{left idealizer} and \textit{right idealizer} are defined as
\[I_L(\cC) =\{X \in \F_{q}^{m \times m} : X C \in \cC \text{ for all } C\in \cC \}, \]
and
\[I_R(\cC) =\{Y \in \F_{q}^{n \times n}: C Y \in \cC \text{ for all }C \in \cC \}, \]
respectively. When $\cC \subseteq \F_{q}^{m \times n}$, $m \leq n$, is a \textit{linear} MRD code with minimum distance $d$, i.e. an optimal code which is a vector space over $\F_q$, it is well known that  $I_L(\cC)$ is a field with $ q \leq |I_L(\cC) |\leq q^m$; moreover, if  $\max\{d, m - d + 2\} \geq  \left \lfloor  \frac{n}{2} \right \rfloor+ 1$, its  right idealizer $I_R(\cC)$ is a field as well, with $q \leq |I_R(\cC)| \leq q^n$, see \cite [Theorem 5.4]{LuTrZh17} and  \cite[Result 3.4]{PolverinoZulloConnections}. Actually, if these are fields, they are isomorphic to subfields of $\F_{q^m}$ and $\F_{q^n}$, respectively.
As a matter of fact, for any subfield $\mathcal M$ of $\F_q^{r\times r}$ containing the identity matrix, $\F_q^r$ is a right vector space over $\mathcal M$ with exterior product $  \mathbf x M$,  for $\mathbf x\in\F_q^r$ and $M\in\mathcal M$.
Then the following holds:
\begin{theorem} \label{idealizersubfields} Let $\cC \subseteq \F_q^{m \times n}$ be a linear rank distance code. If $I_L(\cC)$  (resp.\ $I_R(\cC)$) is a field, then it is isomorphic to a subfield of $\F_{q^m}$ (resp.\ $\F_{q^n})$. \end{theorem}

Given a scattered polynomial $f(x)$,  the set of $q$-polynomials 
\[\mathcal{C}_f=\{ ax+bf(x): a,b\in \F_{q^n} \}= \langle x, f(x) \rangle_{\Fqn} \]
defines a linear MRD code of minimum distance $n-1$ over $\F_q$.
Recall that, given two scattered polynomials $f(x)$ and $g(x)$ over $\F_{q^n}$, the corresponding MRD 
codes $\cC_f$ and $\cC_g$ are \textit{equivalent} if there exist invertible
$L_1$, $L_2\in\homm$ and $\rho\in \Aut(\F_{q})$ 
such that
	\[ L_1\circ \varphi^\rho \circ L_2 \in \cC_g \text{ for all }\varphi\in \cC_f,\]
	where $\circ$ stands for the composition of maps and $\varphi^\rho(x)= \sum a_i^\rho x^{q^i}$ for $\varphi(x)=\sum a_i x^{q^i}$. We denote by $(L_1, L_2, \rho)$ the
equivalence defined above.
In \cite{Sh}, the equivalence between two codes $\cC_f$ and $\cC_g$ is related to the $\GaL$-equivalence of the subspaces $U_f$ and $U_g$, precisely
\begin{theorem} \label{scatteredvscodes}
\cite[Theorem 8]{Sh}  Let $f(x)$ and $g(x)$ be scattered linearized polynomials. Then
$\cC_f$ and $\cC_g$ are equivalent if and only if  $f(x)$ and $g(x)$ are $\GaL$-equivalent. 
\end{theorem} 

If a rank distance code $\cC$ is given as a subset of $\homm$, then the notion of left and right idealizer can be rephrased in this setting as follow
\[I_L(\cC) =\{\varphi\in \homm: \varphi\circ f\in \cC \text{ for all }f\in \cC \} \]
and
\[I_R(\cC) =\{\varphi\in \homm: f\circ \varphi \in \cC \text{ for all }f\in \cC \}, \]
respectively. By \cite[Corollary 5.6]{LuTrZh17} and Theorem \ref{idealizersubfields}, both are isomorphic to subfields of $\F_{q^n}$. Note that any MRD code $\cC_f$ associated with a scattered polynomial $f(x)$ has $I_L(\cC_f)$ isomorphic to the field $\F_{q^n}$.

\begin{proposition} \cite[Lemma 4.1]{LoMaTrZh21}\label{le:full_auto_MRD}
Let $\mathcal{L}_{n,q}$ be the $\Fq$-vector space of all $q$-polynomials with $q$-degree less than $n$. Let $f(x) \in \mathcal{L}_{n,q}$ and  denote $\cC_f$ the associated MRD code.  Then $\Aut(\cC_f)$ 
consists of elements of the type
\[g \mapsto \alpha x^{q^m} \circ g^\sigma \circ L\]
with invertible $L$, $\alpha\in \F_{q^n}^*$, $m \in \{0,1,\ldots, n-1\}$  and $\sigma \in \Aut(\Fq)$ such that $\cC_{f^{\sigma q^m}} \circ x^{q^m}\circ L= \cC_f$.\\
For any $\alpha \in \Fqn^*$, $m \in \{0,1,\ldots,n-1\}$ and $\sigma \in \Aut(\Fq)$, there is a bijection between the set of all $L$ such that  $(\alpha x^{q^m}, L, \sigma) \in \Aut(\mathcal{C}_f)$ and all linear isomorphism between $U_f$ to $U_{f^{\sigma q^m}}$. Furthermore, 
$I_R(\cC_f)^*$ and $G_f$  are isomorphic groups.
\end{proposition}

\begin{remark}\label{isofields}
    It is straightforward to verify that the group isomorphism constructed 
 in \cite[Lemma 4.1]{LoMaTrZh21} can be extended to a field isomorphism between $I_R(\cC_f)$ and $G^\circ_f$. In particular, $|G^\circ_f|=q^t$, where $t$ divides $n$.
\end{remark}

\begin{theorem} \cite[Theorem 2.2]{CsMaPoZh20}
	\label{th:classification}
	 Let $\cC$ be an $\F_q$-subspace of $\mathcal{L}_{n,q}$. Assume that one of the left and right idealizers of $\cC$ is isomorphic to $\F_{q^n}$. Then there exists an integer $k$ such that $|\cC|=q^{kn}$ and $\cC$ is equivalent to 
	\begin{equation}\label{eq:normalized_C_1}
		 \left\{ \sum_{i=0}^{k-1}a_{i} x^{q^{t_i}}+ \sum_{j\notin\{t_0,t_1,\cdots, t_{k-1}\}}  g_j(a_0,\cdots,a_{k-1})x^{q^j}\colon a_0,\cdots,a_{k-1}\in \F_{q^n} \right\}
	\end{equation}
	where $0\le t_0<t_1<\cdots<t_{k-1}\le n-1$ and the $g_j$'s are $\F_q$-linear functions from $\F_{q^n}^k$ to $\F_{q^n}$. If the other idealizer of $\cC$ is also isomorphic to $\F_{q^n}$, then $\cC$ is equivalent to
	\begin{equation*}\label{eq:normalized_C_2}
		\left\{ \sum_{i=0}^{k-1}a_{i} x^{q^{t_i}}\colon a_i\in \F_{q^n} \right\} = \langle x^{q^{t_i}} \, \colon i=0,1,\ldots,k-1 \rangle_{\F_{q^n}}.
	\end{equation*}
\end{theorem}

\begin{proposition}\label{p:leftid}
If $f(x)\in\Fqn[x]$ is a scattered polynomial not $\GaL$-equivalent to a polynomial of pseudoregulus type,
then $G_f^\circ$ is a subring of $\Fqn^{2\times2}$ isomorphic to
a proper subfield $\F_{q^t}$ of $\Fqn$.
\end{proposition}
\begin{proof}
By Proposition \ref{le:full_auto_MRD} the subring $G_f^\circ$ of $\Fqn^{2 \times 2}$ is a field 
isomorphic to the right idealizer of the MRD code $\mathcal{C}_{f}$ associated with $f$. If this is not 
isomorphic to a proper subfield of $\Fqn$, then by Theorem	\ref{th:classification}, $\cC_f$ is equivalent to the MRD $\langle x, x^{q^\ell} \rangle_{\Fqn}$ for some $\ell \in \{0,1,\ldots,n-1\}$ with $(\ell,n)=1$. By Theorem \ref{scatteredvscodes}, $f$ is equivalent to a scattered polynomial of pseudoregulus type, a contradiction. 
\end{proof}

\subsection{Stabilizers of all known maximum scattered subspaces}

In this section we present all known scattered polynomials along with the stabilizers
of their related subspaces, up to one case that we will deal with  in Section \ref{s:sf}.
As usual, $\N_{q^n/q^t}(x)=x^{(q^n-1)/(q^t-1)}$, $t$ a divisor of $n$, denotes the \emph{norm} 
of $x\in\Fqn$ over $\Fqt$. 
We will identify any nonsingular matrix $A\in\Fqn^{2\times2}$ with the map
\[ (x,y)\mapsto (x,y)A \] in $\GL(2,q^n)$.

\medskip

\noindent \textbf{1}. $f(x)=x^{q^s}$, $(s,n)=1$ \emph{(pseudoregulus type)}. In this case 
\[ G_f=\{\diag(\alpha,\alpha^
{q^s})\colon\alpha\in\Fqn^*\}. \]
Cf.\ \cite{CsMaPoZa18}.

\medskip

\noindent \textbf{2}. $f(x)=x^{q^s}+\delta x^{q^{n-s}}$, $(s,n)=1$, $n>3$, $\N_{q^n/q}(\delta)\neq0,1$ \emph{(Lunardon-Polverino type)}. 
For even $n$:
\[ G_f=\{\diag(\alpha,\alpha^{q})\colon\alpha\in\F_{q^2}^*\}. \]
For odd $n$:
\[ G_f=\{\diag(\alpha,\alpha)\colon\alpha\in\Fq^*\}. \]
Cf.\ \cite{CsMaPoZa18}.

\medskip

\noindent \textbf{3}. $f(x)=\delta x^{q^s}+x^{q^{s+n/2}}$, 
$ n \in \{6, 8\}$, $(s, n/2) = 1$, $\N_{q^n/q^{n/2}} (\delta) \not \in 
\{0, 1\}$, with some conditions on $\delta$ and $q$.
In this case 
\[ G_f=\{\diag(\alpha,\alpha^{q^s})\colon\alpha\in\F_{q^{n/2}}^*\}. \]
Cf.\ \cite{CsMaPoZa18}.

\medskip

\noindent \textbf{4}. $f(x)=x^{q}+x^{q^{3}}+\delta x^{q^{5}} \in \F_{q^6}[x]$,  with 
$\delta^2+\delta=1$ for $q$ odd;  some conditions on $\delta$ and $q$, for $q$ even.
In this case 
\[ G_f=\{\diag(\alpha,\alpha^{q})\colon\alpha\in\F_{q^{2}}^*\}. \]
Cf.\ \cite{scatteredeven, CsMaZu18,MaMoZu20}.

\medskip

\noindent \textbf{5}.
\begin{equation}\label{eq:psi}
\psi_{h,t,s}(x)= x^{q^s}+x^{q^{s(t-1)}}+h^{1+q^s}x^{q^{s(t+1)}}+ h^{1-q^{s(2t-1)}}x^{q^{s(2t-1)}},
\end{equation}
$n = 2t$, $t \geq 3$, $(s,n)=1$, $q$ odd, $\N_{q^n/q^t}(h)=-1$.
Cf.\ \cite{LoMaTrZh21,LoZa21,NeSaZu21}.
The stabilizer is described below for $t>4$.
We add a proof since the description in the quoted works is not completely explicit.

\begin{proposition}\cite{LoMaTrZh21,NeSaZu21}\label{p:Gpsi}
Let $\psi_{h,t,s}$ be the scattered linearized polynomial in \eqref{eq:psi}. 
Assume $t>4$. Then
\begin{equation*}
   G_{\psi_{h,t,s}}=\biggl \{ \begin{pmatrix}
\alpha & 0 \\
0 & \alpha^{q}
\end{pmatrix}
\colon \alpha \in \F_{q^2}^* \biggr \}
\end{equation*}
if $t$ is even, and
\begin{equation}\label{eq:statodd}
G_{\psi_{h,t,s}}=\biggl \{ \begin{pmatrix}
\alpha & \xi(h^{q^s}+h^{q^{s(t-1)}})\\
\frac{\xi}{h^{q^s}+h^{q^{s(t-1)}}} & \alpha
\end{pmatrix}
\colon \alpha \in \F_{q},\ \xi^{q^s}+\xi=0 \ \textnormal{and} \  (\alpha,\xi) \neq (0,0) \biggr\}
\end{equation}
if $t$ is odd. In particular for $h \in \Fqt$
\begin{equation*}\label{eq:statoddh}
G_{\psi_{h,t,s}}=\biggl \{ \begin{pmatrix}
\alpha & -4\eta\\
\eta & \alpha
\end{pmatrix}
\colon \alpha \in \F_{q},\ \eta^{q^s}+\eta=0 \ \textnormal{and} \  (\alpha,\eta) \neq (0,0) \biggr\}.
\end{equation*}
\end{proposition}

\begin{proof}

Let 
$\begin{pmatrix}
    \alpha & \beta \\
    \gamma & \delta 
\end{pmatrix}$
be an element in $G_{\psi_{h,t,s}}$. Then
\begin{equation*}
    \beta x+\delta\psi_{h,t,s}=\psi_{h,t,s}(\alpha x)+\psi_{h,t,s}(\gamma\psi_{h,t,s}(x))
\end{equation*}
As in the proof of \cite[Proposition 4.13]{NeSaZu21} and \cite[Theorem 4.2]{LoMaTrZh21}, one obtains
that $\alpha \in \F_{q^2}$ and $\beta=\gamma= 0$, if $t$ is even, and $\alpha \in \Fq$ and 
\begin{equation}\label{eq:gamma}
\gamma =\frac{\xi}{h^{q^s} - h^{-q^{(2t-1)s}}},
\end{equation}
where $\xi^{q^s}+\xi=0$, if $t$ is odd.
Matching the coefficient of $x$ in both sides, one gets
$$\beta=\gamma^{q^s}h^{q^s-1}-\gamma^{q^{s(t-1)}}h^{q^{s(t-1)}-1}-h^{1-q^{s(t+1)}}\gamma^{q^{s(t+1)}}+h^{1-q^{s(2t-1)}}\gamma^{q^{s(2t-1)}}.$$
Then, by \eqref{eq:gamma}, one obtains
\begin{equation*}
\begin{split}
    \beta&=-\xi \biggl (\frac{h^{q^s}}{h^{q^{2s}+1}-1}+\frac{h^{q^{s(t-1)}}}{-1-h^{1-q^{s(t-2)}}}+\frac{h^{q^s-q^{st}}}{h^{q^{s(t+2)}}-h^{-q^{st}}}+\frac{h^{q^{s(t-1)}-q^{st}}}{h-h^{-q^{s(2t-2)}}}\biggr )\\
    &=-\xi \biggl (\frac{h^{q^s}}{h^{q^{2s}+1}-1}+\frac{h^{q^{s(t-1)}}}{h^{q^{s(2t-2)}+1}-1}+\frac{h^{q^s}}{h^{q^{st}+q^{s(t+2)}}-1}+\frac{h^{q^{s(t-1)}}}{h^{q^{s(t-2)}+q^{st}}-1}\biggr )\\
    &=\xi(h^{q^s}+h^{q^{s(t-1)}}).    \end{split}
\end{equation*}
\end{proof}

\begin{remark}\label{t=3,4}
Proposition \ref{p:Gpsi} can be extended to $t\ge3$. Indeed, for $t=3$, it is enough to note that any matrix of the set in the right-hand side of Formula \eqref{eq:statodd} stabilizes $U_{\psi_{h,3,s}}$. Then the stabilizer in $\GL(2,q^n)$ 
 of $U_{\psi_{h,3,s}}$ coincides with this set, since it has to be  a proper subfield of $\F_{q^6}$ (cf.\ Proposition~\ref{p:leftid}) and contains a matrix field of order $q^2$.
For $t=4$, see Remark~\ref{r:ext}.
\end{remark}

\begin{remark}
    Proposition \ref{p:Gpsi}, Remark~\ref{t=3,4} and ~\ref{r:ext} extend \cite[Proposition 7.8, Corollary 7.9 and 7.10]{BarZiZu} for $h \in \F_{q^n} \setminus \F_{q^t}$.
\end{remark}

\subsection{Simultaneous diagonalization}

In \cite{willett}, Willett specialized the results in \cite{beard1,beard2} and characterized 
all subrings $\mathcal{M}$ of $r \times r$ matrices over a finite field which are fields. 
We will briefly recall the main results in \cite{willett}.\\
Let $\Fq^{r \times r}$ be the ring of all square matrices of order $r$ over the finite field $\Fq$, $q=p^e$ and consider $\mathcal{F}_{q,r}$ the collection of all subsets of $\Fq^{r \times r}$ which are fields with the matrix addition and multiplication inherited from $\Fq^{r \times r}$. Clearly, since the scalar matrices of $\Fq^{r \times r}$ form a matrix field, $\mathcal{F}_{q,r}$ is not empty.\\
Note that a matrix field $\mathcal{M}$ does not need to have the identity matrix $I_r \in \F_{q}^{r \times r}$ as its own identity, and hence the
invertibility of a matrix in $\cM$ is not equivalent to non-singularity.
Moreover, each nonzero matrix in $\mathcal{M}$ has the same
rank of  its identity element.

A monic polynomial
\begin{equation*}
    f(x)=x^k-a_1x^{k-1}-\ldots-a_k \in \Fq[x]
\end{equation*}
is called \textit{primitive} if it has a primitive element of the extension $\F_{q^{k}}$ as root.

\begin{theorem}\cite[Theorem 1]{willett}
Let $\mathcal{M} \subseteq \F_{q}^{r\times r}$,
$|\mathcal{M}|=p^k$, for some $k \geq 1$. Then $\mathcal{M}$ is a field if and only if 
\begin{equation}\label{Mshape}
    \mathcal{M}=\{A^i \colon  1 \leq i \leq p^k -1\} \cup \{O\}
\end{equation}
for some matrix $A$ which is similar over $\Fq$ to a matrix  of the form $\mathrm{diag}(0,\ldots,0,D)$, with $D \in \Fq^{\ell \times \ell}$ satisfying $f(D)=O$, where
$f(x)=x^k-a_1x^{k-1}-\ldots-a_k$ is a primitive polynomial over $\F_{p}$.
\end{theorem}
\begin{lemma} \cite[Lemma 3]{willett}\label{factorization}
Let $f(x)$ be a monic irreducible polynomial of degree $k$ over $\F_p$ and let $d= (k,e)$. Then $f(x)$ factors into $d$ irreducible polynomials $f_i(x)$, $0 \leq i \leq d-1$, $\deg f_i(x)=k/d$, over $\Fq=\F_{p^e}$.
\end{lemma}
\begin{theorem}\cite[Theorem 2]{willett}\label{similardiag}
If $D$ is an $\ell \times \ell$ matrix over $\Fq$ such that $f(D)=O$ where $f(x)$ is a monic irreducible polynomial of degree $k$ over $\F_p$, then $D$ is similar over $\Fq$ to a matrix of the form
\begin{equation}\label{diag}
    \mathrm{diag}(D_{n_0},D_{n_1},\ldots,D_{n_\tau}) \quad \quad 0 \leq n_i \leq d-1
\end{equation}
with $D_i=\mathrm{diag}(C(f_i),C(f_i),\ldots,C(f_i))$, where $C(f_i)$ is the companion matrix of the irreducibile factor $f_i(x)$ of $f(x)$.
\end{theorem}

\begin{theorem}\label{t:subfield}
Let $\cM$ be a matrix field in $\Fqn^{2\times2}$ isomorphic to $\F_{q^t}$, $1 \leq t \leq n$,
$t$ a divisor of $n$.
If $\mathcal{M}$ contains a nonsingular matrix, then there is $P\in\GL(2,q^n)$ such that 
\begin{equation}\label{eq:subfield}
P\cM P^{-1}=\left\{\begin{pmatrix}x&0\\ 0& x^ \sigma \end{pmatrix}\colon x\in\F_{q^t}\right\},
\end{equation}
where $\sigma \in \Aut(\Fqt)$. In addition if $\diag(x,x) \in \mathcal{M}$ for any $x \in \Fq$, then $\sigma \in \Gal(\Fqt| \Fq)$.
\end{theorem}
\begin{proof}
Since $\mathcal{M} \in \mathcal{F}_{q^n,2}$ and all its matrices are invertible, the matrix $A$ in \eqref{Mshape} is similar over $\Fqn$ to a matrix $D \in \Fqn^{2 \times 2}$ satisfying $f(D)=O$ where $f(x)$ is a primitive polynomial over $\F_p$  and its degree is $k=et$. Now, by Lemma \ref{factorization}, $f(x)$ factors  in $\Fqn$ into $ (et,en)=et=k$ polynomials of degree 1. Then $f_i(x)=x-\omega^{p^i}$ with $ 0 \leq i \leq k-1$ and $\omega$ is a primitive element of $\Fqt^*$.\\
By Theorem \ref{similardiag}, the matrix $D$ is similar over $\Fqn$ to a matrix of form as in \eqref{diag}. We have to distinguish two cases:
\begin{itemize}
    \item [$i)$] $\tau=0$, then $D$ is similar over $\Fqn$ to a matrix \[D_{n_0}=\diag(C(f_{n_0}),C(f_{n_0}))=\mathrm{diag}(\omega^{p^{n_0}},\omega^{p^{n_0}})\]
    and hence there exists a matrix $P \in \GL(2,q^n)$ such that 
    \[P\mathcal{M} P^{-1}=\{\mathrm{diag}(x,x) \colon x \in \Fqt\}.\]
    \item [$ii)$] $\tau=1$, then $D$ is similar over $\Fqn$ to a matrix
    \[\mathrm{diag}(C(f_{n_0}),C(f_{n_1}))=\mathrm{diag}(\omega^{p^{n_0}},\omega^{p^{n_1}}).\]
    Denoted by $\sigma$ the map $x \in \Fqt \longmapsto x^{p^{n_1-n_0}} \in \Fqt$ and noted that \[\mathrm{diag}(\omega^{ip^{n_0}},\omega^{ip^{n_1}})=\mathrm{diag}(x, x^{\sigma})\] for  $x=\omega^{ip^{n_0}}$, one gets that
    there exists a matrix $P \in \GL(2,q^n)$ such that \[P\mathcal{M} P^{-1}=\{\mathrm{diag}(x,x^{\sigma}) \colon x \in \Fqt\}.\]\end{itemize}
  Finally, if $\diag(x,x) \in \mathcal{M}$ for any $x \in \Fq$, then $\mathcal{M}$ is an $\Fq$-algebra and the automorphism $\sigma$ belongs to $\Gal(\Fqt| \Fq)$.
\end{proof}

Moreover, by \eqref{eq:subfield}, the matrix $P$ diagonalizes all matrices of the matrix field $\mathcal{M}$ and its rows are linearly independent eigenvectors $v_1,v_2 \in \F_{q^n}^2 \setminus \{(0,0)\}$ of all
matrices in $\mathcal{M}$.
{\begin{example}\label{ex:P}
Let $t=n/2$ be an odd integer, $t \geq 3$ and $\psi(x)=\psi_{h,t,s}(x)$ as in \eqref{eq:psi}.
Let $\theta=h^{q^s}+h^{q^{s(t-1)}}$ and \[ P=\begin{pmatrix}1&\theta\\ 1&-\theta\end{pmatrix}. \]
By \eqref{eq:statodd}
\begin{align*}
PG_\psi P^{-1}&=\left\{P\begin{pmatrix}
\alpha & \xi \theta\\
\xi/\theta & \alpha
\end{pmatrix}P^{-1}
\colon \alpha \in \F_{q}, \,\xi^{q^s}+\xi=0 \ \textnormal{and} \  (\alpha,\xi) \neq (0,0) 
\right\}=\\
&=\left\{\begin{pmatrix}\alpha+\xi &0\\
0&\alpha - \xi \end{pmatrix}\colon \alpha \in 
\F_{q},\, \xi^{q^s}+\xi=0 \ \textnormal{and} \  (\alpha,\xi) \neq (0,0)   \right\}.
\end{align*}
Noting
\[
  (\alpha +\xi)^q=\alpha -\xi
\]
leads to 
\[
  PG_\psi P^{-1}=\left\{\begin{pmatrix}a&0\\ 0&a^q\end{pmatrix}\colon a\in\F_{q^2}^*\right\}.
\]
In conclusion, $G_\psi^\circ$ is a subring of $\Fqn^{2\times2}$ isomorphic to $\F_{q^2}$.
\end{example}}

We prove the following result for future reference.

\begin{proposition}\label{p:eigenvectorsnotbelong}
Let $P_1=\langle v_1 \rangle_{\Fqn}$ and $P_2=\langle v_2 \rangle_{\Fqn}$ be points of $\PG(1,q^n)$,
and let $f(x)\in\Fqn[x]$ be a scattered polynomial. 
If $v_1$ and $v_2$ are eigenvectors of all matrices in $G_f$ and $G_f^\circ$ is not isomorphic
to $\Fq$, then $P_i \not \in L_f$ for $i=1,2$.
\end{proposition}
\begin{proof}
Let $t>1$, $t \mid n$ and $|G_f| = q^t-1$. Since $f(x)$ is scattered, 
$$G_f \cap \{\diag(x,x) : x \in \Fqn\}= \{ \diag(x,x) : x \in \Fq\},$$ 
and so $G_f$ induces a group $\tilde{G}_f$ in $\PGL(2,q^n)$ of order  
$(q^t-1)/(q-1)$.
Let $\varphi \in \tilde{G}_f$, if $\varphi$ fixes $P_1$, $P_2$ and a point 
$X \in L_f \setminus \{P_1,P_2\}$, then $\varphi=\mathrm{id}$. 
Hence the orbit of such a point under the action of $\tilde{G}_f$ has order $(q^t-1)/(q-1)$. 
Then this number divides the size of $L_f \setminus \{P_1,P_2\}$. 
Since  $(q^t-1)/(q-1)$ divides  $|L_f|$ as well, if  $ 1 \leq |L_f \cap \{P_1,P_2\}| \leq 2$, 
then $(q^t-1)/(q-1) \in \{1,2\}$, a contradiction.
\end{proof}

\section{Standard form}\label{s:sf}

From now on $\cS_{n,q}$ will denote the set of all scattered polynomials $f(x)\in\Fqn[x]$
such that $G_f^\circ$ is not isomorphic to $\Fq$. 

\begin{definition}
Let $h(x)=\sum_{i=0}^{n-1}b_ix^{q^i}$ be a scattered polynomial,
\[\Delta_h=\{(i-j) \mod n\colon b_ib_j\neq0 \, \text{and} \, i \neq j \}\cup\{n\},\] and let $t_h$ be the greatest common divisor
of $\Delta_h$.
If $t_h>1$ then $h(x)$ is in \emph{standard form}.
\end{definition}
For instance, if $h(x)=x^q+\delta x^{q^{n-1}}\in\Fqn[x]$, $\N_{q^n/q}(\delta)\neq0,1$, $n$ even,
then $\Delta_h=\{2,n-2,n\}$ and $t_h=2$.
So, $h(x)$ is in standard form. On the other hand if $n$ is odd, the same $h(x)$ is not in standard form.

\begin{remark}
If $h(x)$ is in standard form, then
\begin{equation}\label{eq:sf}
    h(x)=\sum_{j=0}^{n/t-1}c_jx^{q^{s+jt}},
\end{equation}
where $t=t_h$ divides $n$, and $0\le s<t$.
More precisely, $s$ is coprime with $t$, otherwise $h(x)$ would be $\F_{q^r}$-linear for $r=(s,t)$ 
contradicting the property to be scattered.
\end{remark}

\begin{theorem}\label{t:sf}
Let $h(x)$ be  a scattered polynomial over $\Fqn$.
The following statements are equivalent:
\begin{enumerate}[$(i)$]
\item $|G_h^\circ|=q^T$, $T>1$, and all elements of $G_h$ are diagonal;
\item $h(x)$ is in standard form.
\end{enumerate}
If the conditions $(i)$, $(ii)$ above hold, then $T=t_h$ and
\begin{equation}\label{eq:sfgf}
G_h^\circ=\biggl \{\begin{pmatrix}
    \alpha & 0\\
    0 & \alpha^{q^s}
\end{pmatrix} \colon \alpha \in \F_{q^{T}} \biggr\}
\end{equation}
where $s$ is as in \eqref{eq:sf}.
\end{theorem}
\begin{proof}
Assume $(i)$.
Then, by Remark~\ref{isofields}, $T$ divides $n$. By Theorem~\ref{t:subfield}, $PG_h^\circ P^{-1}$ is in the form \eqref{eq:subfield} for some nonsingular
matrix $P$, and this implies that $G_h^\circ$ is the set of all $\diag(\alpha,\alpha^{q^s})$
with $\alpha\in\F_{q^T}$ for some integer $s$, $0\le s<T$.
Let $\omega\in\F_{q^T}$ such that $\Fq(\omega)=\F_{q^T}$.
Since $G_h$ stabilizes the scattered subspace $U_h$, one gets that
\[h(\omega x)=\omega^{q^s} h(x) \quad \mbox{ for all } x \in \Fqn.\]
This implies that if $h(x)=\sum_{i=0}^{n-1}b_ix^{q^{i}}$, then
\[b_i\omega^{q^i}=b_i\omega^{q^s} \quad \quad \mbox{ for all }  i=0,1,\ldots,n-1.\]
So, if $b_i \not =0$, we have $\omega^{q^{n+i-s}}=\omega$ and $\omega \in \F_{q^{ (T,i-s)}}$. 
Then $T | (i-s)$, so $h(x)$ is in standard form.
Note that $T$ divides $t_h$.

Next assume $(ii)$, and so \eqref{eq:sf} holds with $t=t_h>1$ and $(s,t)=1$.
Then it can be directly checked that 
$\diag(\alpha,\alpha^{q^s})\in G_h$ for all $\alpha\in\F_{q^{t}}^*$.
If $\alpha\in\F_{q^{t}}\setminus\Fq$, then the only eigenvectors of $\diag(\alpha,\alpha^{q^s})$
are in $\la(1,0)\ra_{\Fqn}$ and $\la(0,1)\ra_{\Fqn}$.
Since the elements of $G_h$ are simultaneously diagonalizable, such eigenvectors are common to all 
matrices of $G_h$. 
Also, $t_h\le T$, that together with $T|t_h$ gives $T=t_h$.
\end{proof}
\begin{remark}\label{r:ext}
Since for even $t$ the polynomial $\psi_{h,t,s}(x)$ is in standard form, 
it follows from Theorem~\ref{t:sf} that Proposition~\ref{p:Gpsi} holds also for $t=4$.
\end{remark}
\begin{theorem}
Any scattered polynomial in standard form is bijective.
\end{theorem}
\begin{proof}
A scattered polynomial in standard form is $\F_{q^t}$-semilinear where $t > 1$, this implies that its kernel is an $\F_{q^t }$-subspace of $\F_{q^n}$ and so it must be bijective.
\end{proof}

\begin{remark}
Note that if $h(x)$ is in standard form, then $h(x)=g(x^{q^s})$ where $g(x)$ 
is an $\Fqt$-linearized polynomial. 
Such $g(x)$ is an \rm{R}-$q^s$-partially scattered polynomial according to the definition given 
in \cite{partially}. 
Indeed, suppose
\begin{equation*}
    \frac{g(y)}{y}=\frac{g(z)}{z}, \quad \frac{y}{z} \in \F_{q^s}
\end{equation*}
with $y,z \in \Fqn^*$. Since $g(x)$ is injective, putting $y=y_0^{q^s}$ and $z=z_0^{q^s}$, one gets
\begin{equation*}
    \frac{g(y^{q^s}_0)}{g(z_0^{q^s})}=\frac{y}{z}=\frac{y_0}{z_0}
\end{equation*}
and so 
\begin{equation*}
    \frac{h(y_0)}{y_0}=\frac{h(z_0)}{z_0}.
\end{equation*}
Since $h(x)$ is scattered, we have $y_0/z_0=y/z \in \Fq$.
\end{remark}

\begin{corollary}
Let $f(x)$ be  a scattered polynomial in $\cS_{n,q}$. 
Then $f(x)$ is $\GL$-equivalent to a polynomial $h(x)$ in standard form.
\end{corollary}
\begin{proof}
There is a nonsingular matrix $P$ such that $PG_f^\circ P^{-1}$ is equal to the right-hand side of
\eqref{eq:sfgf} for some integer $s$.
The map $\varphi: X \mapsto  XP^{-1}$ maps $U_f$ into an $n$-dimensional $\Fq$-subspace $U$.
The stabilizer of $U$ in $\GL(2,q^n)$ is the multiplicative group of the field \eqref{eq:sfgf}.
Therefore, if $(0,y)\in U$ for some $y\in\Fqn^*$, then $(0,\alpha y)\in U$ for all $\alpha\in\Fqt$, contradicting the scatteredness of $f(x)$.
As a consequence,
$U=U_h$, with $h(x)$ a  
scattered $\Fq$-linearized polynomial, and $G_h=PG_fP^{-1}$.
The assertion follows from Theorem~\ref{t:sf}.
\end{proof}
We will refer to the scattered polynomial $h(x)$ above as the \emph{standard form of $f(x)$}.
Only polynomials with a stabilizer not isomorphic to the multiplicative
group of $\Fq^*$ have a standard form.

The next result follows again from Theorem~\ref{t:subfield}, taking into account
that in the case of $G_h$ the automorphism $\sigma$ is not the identity in $\Fqt$,
and means that the standard form is \textit{essentially unique}.
\begin{proposition}
If $h(x)$ and $h'(x)$ are two scattered polynomials in standard form and are $\GL$-equivalent to a scattered 
polynomial $f(x)$, then there are $a,b\in\Fqn^*$ such that $h'(x)=ah(bx)$, or $h'(x)=ah^{-1}(bx)$.
\end{proposition} 

\begin{proof}
Since $h(x)$ and $h'(x)$ are $\GL$-equivalent, a relation $U_{h'}=U_hP$ with $P\in\GL(2,q^n)$ holds.
Since both $G_{h'}^\circ$ and $G_h^\circ=PG_{h'}^\circ P^{-1}$ consist solely of diagonal matrices,
having with the exception of the scalar matrices distinct eigenspaces of dimension one,
either $P=\diag(b^{-1},a)$, or
\[
P=\begin{pmatrix}0&a\\ b^{-1}&0\end{pmatrix}
\]
for some $a,b\in\Fqn^*$.
In the first case,  $U_{h'}=U_hP$ implies that for any $x\in\Fqn$ an $y\in\Fqn$ exists
such that $(x,h'(x))=(b^{-1}y,ah(y))$, implying $h'(x)=ah(bx)$.
In the latter case, for any $x\in\Fqn$ an $y\in\Fqn$ exists
such that $(x,h'(x))=(b^{-1}h(y),ay)$, implying $h'(x)=ah^{-1}(bx)$.
\end{proof}

Since any polynomial that is $\GaL$-equivalent to $f(x)$ is $\GL$-equivalent to $f^\sigma(x)$ for
some automorphism $\sigma$ of $\Fqn$, we have
\begin{proposition}
If $h(x)$ is a scattered polynomial in standard form $\GL$-equivalent to a scattered $q$-polynomial $f(x)$,
then $h^\sigma(x)$ is a scattered polynomial in standard form $\GL$-equivalent to $f^\sigma(x)$.
\end{proposition}
So, two scattered polynomials in $\cS_{n,q}$  having $h(x)$ and $h'(x)$ as
standard forms are $\GaL$-equivalent if, and only if, there are $a,b\in\Fqn^*$ and an
automorphism $\sigma$ of $\Fqn$, such that $h'(x)=ah^\sigma(bx)$, or $h'(x)=a(h^{-1})^\sigma(bx)$.

\begin{example}
Assume $q\equiv1\pmod4$, $t\ge3$ odd, $n=2t$, $h\in\Fq$, $(s,t)=1$, $\psi(x)=\psi_{h,t,s}$. It holds
\begin{equation}\label{eq:bp}
  \psi(x)=x^u+x^{u^{t-1}}-x^{u^{t+1}}+x^{u^{2t-1}},\quad u=q^s.
\end{equation}
The remainder of this section is devoted to find a standard form for $\psi(x)$.
It can be directly checked that
\[
\psi^{-1}(x)=\frac14(x^u-x^{u^{t-1}}+x^{u^{t+1}}+x^{u^{2t-1}}).
\]
Since $q\equiv1\pmod4$, a $\rho\in\Fq$ exists such that $\rho^2=-1$.
Two eigenvectors of all nonscalar matrices in $G_\psi$ (cf. Proposition \ref{p:Gpsi}) are $(1,2\rho)$, $(1,-2\rho)$.
Then a standard form for $\psi(x)$ will be a $q$-polynomial $h(x)$ such that
\[
  U_\psi\begin{pmatrix}2\rho&2\rho\\ -1&1\end{pmatrix}=U_h.
\]
Since
\[
  (x,\psi(x)) \begin{pmatrix}2\rho&2\rho\\ -1&1\end{pmatrix}=(2\rho x-\psi(x),2\rho x+\psi(x)),
\]
one has to find the inverse of the map $2\rho x-\psi(x)$.

Let $h_1(x)=x+2\rho\psi^{-1}(x)$.
This map satisfies $h_1(2\rho x-\psi(x))=-2x^{u^{2t-1}}-2x^u$.

Next, define $h_2(x)=\frac14\sum_{i=1}^t(-1)^ix^{u^{2i-1}}$.
This map satisfies $h_2(-2x^{u^{2t-1}}-2x^u)=x$.
Therefore, $h_2\circ h_1$ is the inverse of $2\rho x-\psi(x)$.

A standard form for $\psi(x)$ is then
\begin{align*}
h(x)&=h_2\circ h_1(2\rho x+\psi(x))=\\
&=h_2\left(4\rho x+\psi(x)-4\psi^{-1}(x)\right)=
h_2\left(4\rho x+2x^{u^{t-1}}-2x^{u^{t+1}}\right)=\\
&=\frac12\sum_{i=1}^t(-1)^i\left(2\rho x+x^{u^{t-1}}-x^{u^{t+1}}\right)^{u^{2i-1}}.
\end{align*}
Since
\[
\sum_{i=1}^t(-1)^i\left(x^{u^{t-1}}-x^{u^{t+1}}\right)^{u^{2i-1}}=
2\sum_{i=1}^{t-1}(-1)^{i+1}x^{u^{t+2i}},
\]
one obtains
\[
h(x)=\rho\sum_{i=1}^t(-1)^ix^{u^{2i-1}}+\sum_{i=1}^{t-1}(-1)^{i+1}x^{u^{t+2i}}.
\]
The form of $h(x)$ is in agreement, by virtue of Theorem~\ref{t:sf},
with the fact that $G_\psi^\circ$ is isomorphic to $\F_{q^2}$.
\end{example}

\begin{example}
Using similar arguments as above one obtains the standard form $H(x)$ for $\psi(x)=\psi_{h,3,s}(x)$.
Let $\theta=h^{q^s}+h^{q^{2s}}$, and 
define \[M=\begin{pmatrix}\theta&\theta\\ 1&-1\end{pmatrix}={2\theta}P^{-1},\] 
where $P$ is the matrix in Example~\ref{ex:P}.
Let $H(x)$ satisfying $U_{\psi}M=U_H$, i.e.
\[ U_H=\{(\theta x+\psi(x),\theta x-\psi(x))\colon x\in\Fqn\}. \]
The map $\theta x+\psi(x)$ is injective and the inverse is up to a factor
\[
  \ell(x)=(-h^{q^s} + h^{1 + q^{2s}})x -h (h^{q^s} + h^{q^{2s}})x^{q^s} 
  -h^{1 + q^s + q^{2s}} (1 + h)x^{q^{3s}}+
  h^{1 + q^{2s}} (h^{q^s} + h^{q^{2s}})x^{q^{5s}}.
\]
The standard form is up to a factor $\ell(\theta x-\psi(x))$, that is
\begin{equation}\label{eq:Hdix}
cH(x)=(1-h^{1+q^{2s}})x^{q^s}+(h+h^2)x^{q^{3s}}+h^{1+q^{2s}}(h+h^{q^s})x^{q^{5s}}\quad (c\in\Fqn^*).
\end{equation}
Note that for $h \in \F_{q^2}$, the trinomial above is equal to that obtained in \cite[Section 3]{BZZ} multiplied by $(h+h^2)$.
\end{example}

\section{Homologies of related translation planes}\label{s:tp}

In order to make this paragraph self-contained, some concepts related to finite 
translation planes are recalled.
The reader is referred to  \cite{An54,Kn95,Lun} for a general treatment of the topic.

A \textit{(planar) spread} of a $(2n)$-dimensional $\Fq$-vector space $\mathbb V$ 
is a collection ${\mathcal F}$ of $q^n+1$ subspaces of $\mathbb V$
of dimension $n$ over $\Fq$, pairwise meeting trivially.
Clearly the union of all elements of ${\mathcal F}$ is $\Fqn^2$.
The geometry $\cA_{{\mathcal F}}$ whose \textit{points} are the elements of $\Fqn^2$ and whose 
\textit{lines} are the cosets of the 
subspaces in the collection ${\mathcal F}$ is an affine translation plane, 
and every finite translation plane
arises in this way.
The elements of ${\mathcal F}$ are called \textit{components} of $\cA_{{\mathcal F}}$ 
and can be regarded both as lines through
the origin, and as points at infinity.

From now on, take $\mathbb V=\F_{q^n}^2$ as $\Fq$-vector space for the construction above.
The affine plane $\cA_{{\mathcal D}}$ associated with the \textit{Desarguesian
spread} 
\[\mathcal{D}= \left \{ \langle v\rangle_{\Fqn} : v \in (\F^2_{q^n})^* \right \} \]
is isomorphic to the affine plane over $\Fqn$.

If a collineation $\kappa$ of a projective plane fixes a line $\ell$ pointwise, then
$\ell$ is an \textit{axis} of $\kappa$; if every line through a point $C$ is fixed setwise,
then $C$ is a \textit{center} of $\kappa$.
A collineation $\kappa$ has an axis if and only if it has a center; in this case, 
if $\kappa$ is not the identity, axis and center are unique. 
A nontrivial collineation having an axis is a \textit{central collineation}.
If the center of a central collineation $\kappa$ belongs to the axis, then
$\kappa$ is called an \textit{elation}; otherwise it is a \textit{homology}.

A collineation of an affine plane is called an \textit{affine central collineation} if its
extension to the projective plane is a central collineation, and the axis is a proper line.
Since the lines fixed by a central collineation are precisely the axis and the lines through
the center, the center of an affine central collineation $\kappa$ is a point at infinity (that is,
a component).
Any affine line through the center is called a \textit{coaxis} of $\kappa$.

The \textit{kernel} of $\cA_{{\mathcal F}}$ is
\[
  K(\cA_{{\mathcal F}})=\{\delta\in\End(\Fqn^2)\colon \delta(W)\subseteq W\mbox{ for any }W\in{\mathcal F}\},
\]
where $\End(\Fqn^2)$ denotes the endomorphism ring of the vector space $\Fqn^2$.
As is well known, $K(\cA_{{\mathcal F}})$ is a field, isomorphic to the kernel of any quasifield
co-ordinatizing $\cA_{{\mathcal F}}$, and trivially contains $\Fq$.
Any $\delta\in K(\cA_{{\mathcal F}})$, $\delta\neq0$, is a homology,
called a \textit{kernel homology} of $\cA_{{\mathcal F}}$.

If $\kappa$ is a collineation of an affine traslation plane $\cA_{{\mathcal F}}$, then there exists
a $K(\cA_{{\mathcal F}})$-semilinear automorphism $\lambda$ of $\Fqn^2$ and an $u\in\Fqn^2$, such that 
$\kappa(v)=\lambda(v)+u$ for all $v\in\Fqn^2$.

Let $\kappa$ be an affine central collineation of  $\cA_{{\mathcal F}}$.
Let $\lambda$ be the $K(\cA_{{\mathcal F}})$-semilinear automorphism related to $\kappa$, as above.
Then there is a component $W\in\mathcal F$ such that the restriction of $\lambda$ to $W$ is the
identity map. 
As a matter of fact, let $a\in\Fqn^2$ and $W\in{\mathcal F}$ such that the restriction of 
$\kappa$ to $a+W$ is the identity map.
By
\[\lambda(a)+\lambda(x)+u=a+x\ \mbox{for any }x\in W\quad\mbox{and }\lambda(0,0)=(0,0), \]
one deduces $\lambda(x)=x$ for any $x\in W$.

Next the procedure described in \cite{tp1} for obtaining a translation plane from a 
scattered $\Fq$-linearized polynomial $f(x)\in\Fqn[x]$ is reported.
In \cite{tp1}, most results are stated for $q>3$, which we will assume from now on.
The property of scatteredness implies that if $h,h'\in\Fqn$ are $\Fq$-linearly independent,
then $hU_f\cap h'U_f=\{(0,0)\}$.
Furthermore, the union of all $\Fq$-subspaces $hU_f$, $h\in\Fqn^*$, is equal to the union
of all subspaces in 
\[
L_f=\{\la(x,f(x))\ra_{\Fqn}\colon x\in \Fqn^*\}
\]
which is a subset of the Desarguesian spread ${\mathcal D}$.
Therefore, 
\[
  \cB_f=\left({\mathcal D}\setminus L_f\right)\cup\{ hU_f \colon h\in\Fqn^*\}
\]
is a spread of $\Fqn^2$, defining a translation plane $\cA_f=\cA_{\cB_f}$.
The kernel of such plane is isomorphic to $\Fq$ \cite{tp1}.
Recall that
\begin{theorem}\cite[Theorem 4.2]{tp1}
If $f(x)\in\Fqn[x]$ is a scattered polynomial, and $q>3$, then 
the $\Fq$-semilinear automorphism $\lambda$ related to any collineation of $\cA_f$
belongs to $\GaL(2,q^n)$, i.e., it is $\Fqn$-semilinear.
\end{theorem}

Let $H_f=\{d\varphi \colon d \in \Fqn^*, \varphi \in G_f\}=\Fqn^*G_f.$
This $H_f$ is the group of all linear collineations of the 
translation plane $\cA_f$ by \cite[Corollary 4.3]{tp1}.
\begin{proposition}\label{p:qlin}
Let $f(x)\in\Fqn[x]$ be a scattered polynomial, and $q>3$, $n>2$.
Let $\kappa$ be an affine central collineation of $\cA_f$, and
let $\lambda$ be the semilinear automorphism related to $\kappa$.
Then the automorphism of $\Fqn$ associated with $\lambda$ is trivial. 
As a consequence, $\lambda\in H_f$.
\end{proposition}
\begin{proof}
Let $q=p^e$, $0\le k<ne$, and define $\tilde x=x^{p^k}$ for any $x\in\Fqn$.
A component of $\cA_f$, say $W$, exists which is pointwise fixed by 
\[
\lambda: (x, y)\ \mapsto\ (\tilde x, \tilde y)\begin{pmatrix}a&b\\ c&d\end{pmatrix},\quad ad-bc\neq0.
\]
Four cases are possible.\\
1) $W=\{(0, y)\colon y\in\Fqn\}$.
This implies $d\tilde y=y$ for any $y\in\Fqn$ hence $k=0$.\\
2) $W=\{(x, mx)\colon x\in\Fqn\}$.
Then $a\tilde x+c\tilde m\tilde x=x$ for any $x\in\Fqn$ leading to the thesis trivially once again.\\
3) $W=hU_f$, $h \in \Fqn^*$, and $L_f$ is not of pseudoregulus type.
Then $a\tilde h\tilde x+c\tilde h\widetilde{f(x)}=hx$  for any $x\in\Fqn$.
Let $f(x)=\sum_{i=0}^{n-1}a_ix^{q^i}$.
It holds
\begin{equation}\label{eq:idpol}
  a\tilde h\tilde x+c\tilde h\sum_{i=0}^{n-1}\tilde{a_i}\tilde x^{q^i}-hx=0 \pmod{x^q-x}.
\end{equation}
Since $L_f$ is not of pseudoregulus type, $\sum_{i=0}^{n-1}\tilde{a_i}\tilde x^{q^i}$
has at least two monomials of distinct degrees not of type $\ell\tilde x\pmod{x^{q^n}-x}$ 
and at least one of them is not of type $\ell'x$ $\pmod {x^{q^n}-x}$.
Hence \eqref{eq:idpol} implies $c=0$ and $\tilde x=x$ for any $x\in\Fqn$.\\
4) $W=hU_f$, $h \in \Fqn^*$ and $L_f$ is of pseudoregulus type.
Then \cite{CsZa} there exists $\varphi\in\GL(2,q^n)$ and $s\in\{1,2,\ldots,n-1\}$, $(s,n)=1$,
such that $(hU_f)^\varphi=R$ where $R=\{(x,x^{q^s})\colon x\in\Fqn\}$.
Therefore $\lambda'=\varphi\circ\lambda\circ\varphi^{-1}$ fixes $R$ pointwise.
The automorphism of $\Fqn$ related to $\lambda'$ is again $x\mapsto\tilde x$.
So
\[
(\tilde x, \tilde x^{q^s})\begin{pmatrix}A&B\\ C&D\end{pmatrix}=(x, x^{q^s})\ \mbox{ for all }
x\in\Fqn,\ AD-BC\neq0.
\]
Assume $k\neq0$.
Then $A\tilde x+C\tilde x^{q^s}=x$ for any $x$ implies $A=0$, $C=1$ and $\tilde x=x^{q^{n-s}}$.
The equation $B\tilde x+D\tilde x^{q^s}=x^{q^s}$ implies then $\tilde x=x^{q^s}$;
so, $x^{q^{n-s}}=x^{q^s}$ for all $x\in\Fqn$, contradicting $(s,n)=1$.
\end{proof}
\begin{remark}
Proposition~\ref{p:qlin} cannot be extended to $n=2$, because $x^{q^{n-s}}=x^{q^s}$ does not contradict $(n,s)=1$. 
For $n=2$, every line of $\cA_f$ that is not also a line of the Desarguesian plane $\cA_{\mathcal D}$ is of type $a+hU_f$, $a\in\F_{q^2}^2$, $h\in \F_{q^2}^*$. This is a Baer subplane of $\cA_{\mathcal D}$ containing the $q+1$ points of $L_f$ at infinity. 
Therefore, $\cA_f$ is a well-known Hall plane \cite[Chapter X]{HuPi73}, and we will not deal with this case.
\end{remark}

Recall that for $f(x)\in\Fqn[x]$ a scattered polynomial, the
\textit{kernel homology group of the associated Desarguesian plane} is the set of all maps
$\lambda_a:(x,y)\mapsto(ax,ay)$, $a\in\Fqn^*$.
Despite the name,
if $a\notin\Fq$, then $\lambda_a$ is a collineation of $\cA_f$, and not a central collineation.
The maps of type $\lambda_a$ with $a\in\Fq^*$ are the kernel homologies of $\cA_f$.

Given two groups $G$ and $H$ of affine homologies and two lines
$\ell$, $m$, if $\ell$ is axis of any element in $G$ and coaxis of any element in $H$,
and furthermore $m$ is axis of any element in $H$ and coaxis of any element in $G$, then
$G$ and $H$ are called \emph{symmetric affine homology groups}.

The following theorem describes the structure of the central collineations of $\cA_f$
and generalizes the result found in~\cite{JhJo08} that deals with  $f(x)$ of Lunardon-Polverino type.
Recall that $\cS_{n,q}$ denotes the set of all scattered
$\Fq$-linearized polynomials in $f(x)\in\Fqn[x]$, such
that $|G_f|>q-1$.
\begin{theorem}\label{t:structure} 
Assume that $f(x)\in\Fqn[x]$ is a scattered polynomial and $q>3$,
$n>2$.
\begin{enumerate}[$(i)$]
\item If $f(x)\notin\cS_{n,q}$, then the plane $\cA_f$ admits no nontrivial affine central collineation
group; the full collineation group in $\GL(2, q^n )$ has order $(q^n- 1)$ and is
the kernel homology group of the associated Desarguesian plane.
\item If $f(x)\in\cS_{n,q}$, and $G_f$ is isomorphic to $\Fqt^*$, the plane $\cA_f$ admits 
{cyclic} symmetric  
affine homology groups of order $(q^t-1)/(q-1)$ but admits no nontrivial affine elation.
Two distinct components $X$ and $Y$ exist such that any 
affine homology has center either in $X$ or in $Y$.
Such components are elements of the Desarguesian spread $\mathcal{D}$ of $\Fqn^2$.
The full collineation group in $\GL(2, q^n )$ is the direct product
of the kernel homology group of the associated Desarguesian plane of order $(q^n-1)$ 
by a cyclic homology group of order $(q^t-1)/(q-1)$.
\end{enumerate} 
\end{theorem}
\begin{proof}
We first determine the full collineation group in $\GL(2, q^n )$, that is, $H_f$.
By Proposition~\ref{p:qlin}, all affine central collinations of $\cA_f$ fixing the origin are in $H_f$.
By Theorem \ref{t:subfield}, $H_f$ is conjugate, by a nonsingular matrix $P$, to a group
\[PH_fP^{-1}=\biggl \{ d\begin{pmatrix}
    \alpha & 0 \\
    0 & \alpha^{q^s}
\end{pmatrix} \colon d \in \Fqn^*,\ \alpha \in \Fqt^* \biggr \},\]
where $t$ divides $n$, and $ (s,t)=1$. 
If $f(x)\notin\cS_{n,q}$ such group is the set of all scalar matrices of the kernel homology group of the Desarguesian plane.

Next, assume $f(x)\in\cS_{n,q}$.
The map $\varphi:X\mapsto XP^{-1}$ is an isomorphism between $\cA_f$ and some translation
plane $\cA'=\cA_{\cB'}$.
The stabilizer of $\cB'$ in $\GL(2,q^n)$ is then $PH_fP^{-1}$.
The map $M\mapsto PMP^{-1}$ maps affine central collineations in $\cA_f$ with axis through the origin
into collineations in $\cA'$ of the same type, and conversely.
If $W$ is an eigenspace of a nonscalar
matrix in $PH_fP^{-1}$, then $W=X'=\langle(1,0)\rangle_{\Fqn}$ or $W=Y'=\langle(0,1)\rangle_{\Fqn}$.
Let $X=X'P$, $Y=Y'P$.
If a nonscalar $\lambda\in H_f$ fixes some points of $\cA_f$, then they are precisely
either those in $X$ or those in $Y$.
Assume $X$ is an eigenspace with eigenvalue one.
Since by Proposition \ref{p:eigenvectorsnotbelong} $X, Y \in\cB_f$, $\lambda$ is a homology of axis $X$ 
and center $Y$. 
The same argument holds exchanging $X$ and $Y$.

Consider the subgroup of $PH_fP^{-1}$
\begin{equation*}
K=\{ \diag(1, \alpha^{q^s-1}) \colon \alpha \in \F_{q^t}^*\}.
\end{equation*}
This is a cyclic homology group of order $(q^t-1)/(q-1)$ generated by an element of the type 
$\diag(1,\omega^{q^s-1})$ where $\omega$ is a primitive element of $\F_{q^t}^*$.  
Since any element in $PH_fP^{-1}$ can be written uniquely in the following way
\begin{equation}\label{eq:fcg}
\begin{pmatrix}
d & 0 \\
0 & d 
\end{pmatrix}
\begin{pmatrix}
1 & 0 \\
0 & \alpha^{q^s-1}
\end{pmatrix},
\end{equation}
the statements regarding structure and size of the collineation group in $\GL(2,q^n)$ follow.
\end{proof}

\begin{remark}
In the case of a linear set of pseudoregulus type, the points $X$ and $Y$ in
Theorem~\ref{t:structure} are known as \emph{transversal points} \cite{LMPT}.
It remains an open problem whether in general these transversal points depend only on the linear set 
$L_f$ or depend on the polynomial representing it.
\end{remark}

\begin{definition}
Let $G$ be a group acting on an abelian group $W$.
If no nontrivial subgroup of $W$ is invariant under the action of $G$, then
\emph{$G$ acts on $W$ irreducibly}.
\end{definition}

\begin{proposition}\cite{Lun}\label{Lun}
Let $\cA$ be a generalized Andr\'e translation plane.
Then there is a group $G$ of collineations of $\cA$, such that for any component $W$ of $\cA$,
with at most two exceptions, $G_{\{W\}}$ acts on $W$ irreducibly.
Furthermore, any element of $G$ is the product of two affine homologies.
\end{proposition}

\begin{corollary}\label{c:necessandr}
Let $f(x)\in\Fqn[x]$ be a scattered polynomial, and assume that $\cA_f$ is a generalized Andr\'e plane. 
Then for any component $W$ of $\cA_f$, with at most two exceptions, $(H_f)_{\{W\}}$ acts on $W$ irreducibly.
\end{corollary}

As it has been proved in \cite{tp1, LP2001}, if $L_f$ is a linear set of pseudoregulus type, then $\mathcal{A}_f$ is an Andrè plane. On the other hand, it holds

\begin{theorem}\label{t:notgap}
Assume that $f(x)\in\Fqn[x]$, $q>3$, is a scattered polynomial, 
and that $L_f$ is not of pseudoregulus type.  
Then $\cA_f$ is not a generalized Andr\'e plane.
\end{theorem}
\begin{proof}
Assume that $\cA_f$ is a generalized Andr\'e plane.
Define $\varphi:X\mapsto XP^{-1}$ as in Theorem~\ref{t:structure}. 
So, by Corollary~\ref{c:necessandr}, for any $W$ in $\cB'$, with at most two exceptions,
$(PH_fP^{-1})_{\{W\}}$ acts irreducibly on $W$.
Since $(q^n-1)/(q-1)>2$, it may be assumed that $(PH_fP^{-1})_{\{W\}}$ acts irreducibly on 
a component of type
\[ W=(hU_f)^\varphi,\quad h\in\Fqn^*. \]

Next, note that
\[ PH_fP^{-1}=\{dPMP^{-1}\colon M\in G_f,\ d\in\Fqn^*\}. \]
Assume $dPMP^{-1}\in(PH_fP^{-1})_{\{W\}}$ for $M\in G_f$, $d\in\Fqn^*$.
For any $x\in\Fqn$ there exists $y\in\Fqn$ such that
\[ dh(x, f(x))MP^{-1}=h(y,f(y))P^{-1}. \]
This implies $d(x,f(x))\in U_f$ and $d\in\Fq$ since $f(x)$ is scattered.
Since $dI_2\in G_f$ for any $d\in\Fq^*$, one deduces that any element of
$(PH_fP^{-1})_{\{W\}}$ is of type $PMP^{-1}$ with $M\in G_f$.
The component $W$ of $\cB'$ is an $n$-dimensional scattered $\Fq$-subspace of $\Fqn^2$.
By Proposition~\ref{p:eigenvectorsnotbelong},
$W\cap(\{0\}\times\Fqn)=\{(0,0)\}$.
Therefore, there is an $\Fq$-linear map $g:\Fqn\rightarrow\Fqn$ such that
\[ W=\{(x,g(x))\colon x\in\Fqn\}. \]
By Proposition \ref{p:leftid}, any matrix in $PG_fP^{-1}$ has 
coefficients in some  $\F_{q^t}$ with $t<n$.
Then, since $\{(x,g(x))\colon x\in\Fqt\}$ is a subgroup of $(W,+)$ invariant under the action of
$(PH_fP^{-1})_{\{W\}}$, $(PH_fP^{-1})_{\{W\}}$ does not act irreducibly on $W$, a contradiction.
\end{proof}

{Theorem \ref{t:notgap} has been proved in \cite{JhJo08} for a polynomial
of Lunardon-Polverino type.}

\noindent
Giovanni Longobardi\\
Dipartimento di Matematica e Applicazioni
 “Renato Caccioppoli”\\
Università degli Studi di Napoli Federico II\\
Via Cintia, Monte S. Angelo
I\\
80126 Napoli - Italy\\
\emph{giovanni.longobardi@unina.it}

\medskip

\noindent Corrado Zanella\\
Dipartimento di Tecnica e Gestione dei Sistemi Industriali\\
Universit\`a degli Studi di Padova\\
Stradella S. Nicola, 3\\
36100 Vicenza VI - Italy\\
\emph{corrado.zanella@unipd.it}

\end{document}